\documentclass[12pt]{amsart}
\usepackage[utf8]{inputenc}
\usepackage[english]{babel}
\usepackage{multirow}
\usepackage{anysize}
\usepackage{color}
\usepackage{multicol}

\usepackage{enumitem}
\usepackage{amsfonts}
\usepackage{amsmath}
\usepackage{amssymb}
\usepackage[arrow, matrix, curve, xdvi, dvips]{xy}
\usepackage{amsthm}

\usepackage{blkarray}
\usepackage{booktabs}
\usepackage{mathtools}

\usepackage{graphicx}
\usepackage[usenames,dvipsnames]{pstricks}
 \usepackage{pstricks-add}
 \usepackage{epsfig}

\newtheorem{teo}{Theorem}[section]
\newtheorem{defn}[teo]{Definition}
\newtheorem{remark}[teo]{Remark}
\newtheorem{cor}[teo]{Corollary}
\newtheorem{example}[teo]{Example}
\newtheorem{proposition}[teo]{Proposition}

\newtheorem{lemma}[teo]{Lemma}
\title[LIE ALGEBRAS ASSOCIATED WITH LABELED DIRECTED GRAPHS]{LIE ALGEBRAS ASSOCIATED WITH LABELED DIRECTED GRAPHS}

\author{Mauricio Godoy Molina and Diego Lagos}

\address{Departamento de Matem\'atica y Estad\'istica, Universidad de La Frontera, Av. Francisco Salazar 01145 Temuco, Chile.}



\email{mauricio.godoy@ufrontera.cl}
\email{diego.lagos@ufrontera.cl}

\thanks{This research is partially supported by Fondecyt \#1181084}

\begin{document}
\maketitle

\begin{abstract}
We present a construction of 2-step nilpotent Lie algebras using labeled directed simple graphs, which allows us to give a criterion to detect certain ideals and subalgebras by finding special subgraphs. We prove that if a label occurs only once, then reversing the orientation of that edge leads to an isomorphic Lie algebra. As a consequence, if every edge is labeled differently, the Lie algebra depends only on the underlying undirected graph. In addition, we construct the labeled directed graphs of all 2-step nilpotent Lie algebras of dimension $\leq6$ and we compute the algebra of strata preserving derivations of the Lie algebra associated with the complete bipartite graph $K_{m,n}$ with two different labelings.
\end{abstract}

\section{Introduction}

Nilpotent Lie algebras have been an important object of study in many areas of mathematics, such as geometric analysis, representation theory and control theory for many years, for example see \cite{abb,bon, cap, j,klr, m}. In particular, nilpotent Lie algebras of step 2 have attracted much interest for some time, not only because they are the simplest ones in most senses, but due to several close connections with Clifford algebras, dynamics and differential geometry \cite{eberlein1, eberlein2,fu,markina, kaplan1,kaplan2}. 

In 2005, S. G. Dani and M. G. Mainkar \cite{mainkar1} proposed a method for constructing 2-step nilpotent Lie algebras using graphs in a very natural manner. This construction is sufficiently rigid so that one can prove very interesting results relating the combinatorial structure of the graph and the algebraic structure of the Lie algebra, for example see  \cite{aaa} for a deep connection with the notion of degeneration. Moreover, this point of view has been generalized by A. Ray in \cite{ray} in a very original manner to certain 3-step nilpotent Lie algebras. 

In the present paper, we develop further the construction of the 2-step nilpotent Lie algebra structures associated to labeled directed graphs introduced in \cite{CMS}. This construction is a generalization from the one presented in \cite{mainkar1}. From the definition of the associated Lie bracket, the structure constants are $0,-1$ and $1$ in an appropriate basis. Examples of these Lie algebras are very important in geometric analysis, see \cite{CD,markina}, since by a well-known result by Mal'cev if a real or complex Lie algebra ${\mathfrak g}$ admits a basis with respect to which the structure constants are rational, then the unique connected and simply connected Lie group ${\mathfrak G}$ that has ${\mathfrak g}$ as its Lie algebra has a lattice, that is, a discrete co-compact subgroup, see \cite{CG,malcev}. These objects provide natural settings for the study of compact nilmanifolds, such as in \cite{bauer}.

The structure of the paper is the following. We present the construction of the 2-step nilpotent Lie algebra associated to a labeled directed graph in Section \ref{sec:const}. Afterward, in Section \ref{sec:ideal}, we prove a combinatorial relation between certain special subgraphs and subalgebras and ideals of the corresponding Lie algebra. In Section \ref{sec:magnin}, we summarize explicitly some of the ideals and subalgebras presented in the previous section in all 2-step nilpotent Lie algebras of dimensions 4, 5 and 6, according to the classification in \cite{magnin}. We conclude the paper with Section \ref{sec:Kmn}, in which we use the combinatorial description of certain Lie algebras associated to complete bipartite graphs to study the structure of the algebra of strata preserving derivations in these examples.

\section{Constructing a Lie algebra from a labeled directed graph}\label{sec:const}

\noindent A finite {\it directed graph} $(V,E)$ consists of a finite set $V=\{x_1,\dotsc,x_n\}$ of vertices and a finite set $E=\{\overrightarrow{x_ix_j}\mid x_i,x_j\in V\}$ of directed edges. We say that $(V,E)$ is {\it simple} if the underlying undirected graph is simple, that is, it has neither loops nor multiple edges.

\noindent The directed graphs we will be interested in this paper are edge-labeled, that is, given a finite simple directed graph $(V,E)$ we will consider a \emph{labeling function}
\[
c\colon E\to \mathcal{C}
\]
where $\mathcal{C}=\{c_1,\dots,c_m\}$ is the set of labels. We denote the label of edge $\overrightarrow{xy}$ by $c(x,y)$.

\begin{defn}
Given a finite simple directed graph $(V,E)$ with a surjective labeling function $c\colon E\to \mathcal{C}$, we will denote it by $G=(V,E,c)$ and call it a \emph{labeled directed graph}.
\end{defn}

Given a labeled directed graph $G=(V,E,c)$, we can associate to it a finite dimensional 2-step nilpotent Lie algebra ${\rm Lie}(G)$ over a field $F$ of characteristic ${\rm char}(F)\neq2$ by means of the following procedure:

\begin{itemize}
\item ${\rm Lie}(G)={\rm span}_F\{x_1,\dotsc,x_n,c_1,\dots,c_m\}$ as a vector space.
\item Define the Lie bracket on the basis $\{x_1,\dotsc,x_n,c_1,\dots,c_m\}$ as follows
\[
[x_i,x_j]_{{\rm Lie}(G)}=\begin{cases}c_l,&\mbox{if }c(x_i,x_j)=c_l,\\-c_l,&\mbox{if }c(x_j,x_i)=c_l,\\0,&\mbox{otherwise,}\end{cases}\quad,\quad [x_i,c_l]_{{\rm Lie}(G)}=[c_l,c_\ell]_{{\rm Lie}(G)}=0,
\]
and then extend it linearly. For the rest of the paper, the indices $i,j$ correspond to vertices and $l,\ell$ correspond to labels.
\end{itemize}
Taking into account the fact that $G$ is simple, this definition includes the identity $[x,x]=0$ needed for Lie algebras over fields of characteristic 2. In all other characteristics, this is equivalent to the skew-symmetry of the Lie bracket. Note that the Jacobi identity is satisfied trivially. We consider only fields of characteristic different from 2 due to technical reasons in some of the proofs below.

According to the construction described before, these algebras are naturally {\em stratified}, that is, they admit a grading
\[
{\rm Lie}(G)={\mathfrak{g}}_{-2}\oplus{\mathfrak{g}}_{-1},\quad {\mathfrak{g}}_{-2}={\rm span}_F\{c_1,\dots,c_m\},\quad {\mathfrak{g}}_{-1}={\rm span}_F\{x_1,\dotsc,x_n\},
\]
for which $[{\mathfrak{g}}_{-1},{\mathfrak{g}}_{-1}]={\mathfrak{g}}_{-2}$. In the cases $F={\mathbb R}$ and $F={\mathbb C}$, Lie algebras satisfying this condition are particularly important, either for geometric, algebraic or analytic reasons, for example, see \cite{bon,nico,ben}.


From the definition of ${\rm Lie}(G)$, with respect to the basis $V\cup{\mathcal C}$ the structure constants are $0,-1$ and $1$. As mentioned in the introduction, in the real and complex cases, these Lie algebras admit a lattice. Moreover, given such an algebra with a basis as before, then one can easily construct an appropriate labeled directed graph.

\begin{remark}
The real Lie algebra $N(d)$ in \cite{sch} with $d$ irrational, is an example of a $2$-step nilpotent Lie algebra cannot be obtained from a labeled directed graph because its structure constants are not rational in any basis. 
\end{remark}

\begin{example}\label{ex:1} Consider the labeled directed graph

\begin{figure}[h]
\includegraphics[scale=0.35]{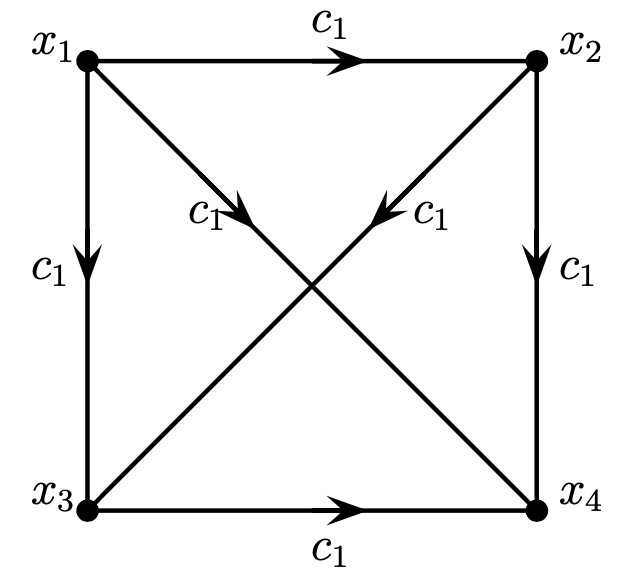}
\caption{Graph in example \ref{ex:1}}
\end{figure}

According to the construction, this graph corresponds to the 5-dimensional 2-step nilpotent Lie algebra ${\mathfrak{g}}={\rm span}\{x_1,x_2,x_3,x_4,c_1\}$ with non-trivial Lie brackets
\[
[x_1,x_2]_{\mathfrak{g}}=[x_1,x_3]_{\mathfrak{g}}=[x_1,x_4]_{\mathfrak{g}}=[x_2,x_3]_{\mathfrak{g}}=[x_2,x_4]_{\mathfrak{g}}=[x_3,x_4]_{\mathfrak{g}}=c_1.
\]

It is important to observe that, in fact, this Lie algebra is isomorphic to the Lie algebra ${\mathfrak{g}}_{5,1}$ in \cite{magnin}, which corresponds to the graph

\begin{figure}[h]
\includegraphics[scale=0.35]{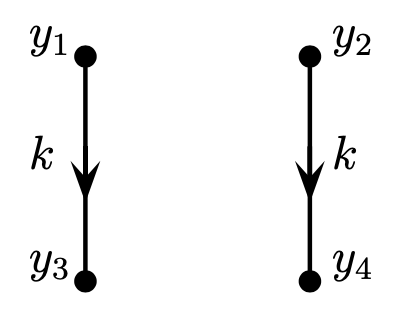}
\caption{Isomorphic algebra with a simpler graph}
\end{figure}

This graph will appear again in this paper in Table \ref{dim45}, when listing the graphs of all low-dimensional real nilpotent Lie algebras. An explicit graded isomorphism $\varphi:{\mathfrak{g}}_{5,1}\to{\mathfrak{g}}$ is given in the bases by
\[
\varphi(y_1)=x_1,\quad\varphi(y_2)=x_1+x_2-x_4,\quad\varphi(y_3)=x_1+x_2+x_3,\quad\varphi(y_4)=x_2,\quad\varphi(k)=c_1.
\]
\end{example}

\begin{remark}
The construction above is different from the ones that can be found in \cite{mainkar1,ray}. The main advantage is that our construction allows for a much broader class of Lie algebras. The main drawback is that we lose some uniqueness results, as seen in the previous example.
\end{remark}

\section{Graph-ideals of ${\rm Lie}(G)$}\label{sec:ideal}

In this section we will study how certain aspects of the combinatorics of the graph $G$ are related to the algebraic structure of the Lie algebra ${\rm Lie}(G)$. In order to proceed initially, we need to recall a couple of concepts of graph theory.

\begin{defn} 
Assume the vertices of an undirected graph $(V,E)$ are numbered as $V=\{x_1,\hdots, x_n\}$.
\begin{itemize}
\item The {\em adjacency matrix}, given by $$\mathcal{A}(V,E)=(a_{ij})_{1\leq i,j\leq n}$$ where  $a_{ij}=1$ if $x_i$ and $x_j$ are connected, and $a_{ij}=0$ otherwise. 
\item The {\em valency matrix}, given by $$\mathcal{B}(V,E)=(b_{ij})_{1\leq i,j\leq n}$$ is a diagonal matrix, where the diagonal entry $b_{ii}$ corresponds to the degree of the vertex $x_i$, that is, the number of vertices connected to $x_i$.
\item The {\em Laplacian matrix} $\mathcal{L}(V,E)$ is the difference $\mathcal{A}(V,E)-\mathcal{B}(V,E)$.
\end{itemize}
\end{defn}

Following well-known results in spectral graph theory, the first natural question to ask concerns the connectedness of the graph (see \cite{chung}).

\begin{proposition}
Let ${\mathfrak g}={\mathfrak g}_{-2}\oplus{\mathfrak g}_{-1}$ be a stratified 2-step nilpotent Lie algebra admitting a stratified basis $\{c_1,\dotsc,c_m\}\cup\{x_1,\dotsc,x_n\}$ with structure constants $0,-1$ and $1$. The associated labeled directed graph $G=(V,E,c)$ has $\dim\ker\mathcal{L}(V,E)$ connected components.
\end{proposition}

\begin{proof}
The connected components of a graph are in one-to-one correspondence with the linearly independent eigenvectors of $\mathcal{L}(V,E)$ with eigenvalue zero (see \cite{nica}). 

The Laplacian matrix $\mathcal{L}(V,E)$ can be easily found from ${\mathfrak g}$, since the adjacency matrix $\mathcal{A}(V,E)$ of the graph can be directly found from the structure constants of ${\mathfrak g}$, forgetting the signs. The result follows.
\end{proof}

We study the relation between certain special subgraphs of $G$ and the subalgebras and ideals of ${\rm Lie}(G)=\mathfrak{g}_{-2}\oplus\mathfrak{g}_{-1}$. Note that any vector subspace of $\mathfrak{g}_{-2}$ is an ideal of ${\rm Lie}(G)$, in particular, the algebra ${\rm Lie}(G)$ is not simple. 

\begin{defn}
The following subalgebras and ideals of ${\rm Lie}(G)$ are considered {\em trivial}:
\begin{itemize}
\item The Lie algebra ${\rm Lie}(G)$.
\item Vector subspaces of $\mathfrak{g}_{-2}$.
\item Abelian factors, that is, spanned by disconnected single vertices in $G$.
\end{itemize}
\end{defn}

Recall that a subgraph $G'$ of a directed graph $G$ is {\em induced} if it is formed from a subset of the vertices of $G$ and all of the edges of $G$ connecting pairs of vertices in that subset.

%

\begin{proposition}\label{ideals}
Let $G=(V,E,c)$ be a labeled directed graph and $G'=(V',E',c)$ a induced subgraph of $G$. Then

\begin{enumerate}
	\item ${\rm Lie}(G')$ is a subalgebra of ${\rm Lie}(G).$
	
	\item If for each pair of adjacent vertices $x,y$ of $G'$ we have $[x,y]_{{\rm Lie}(G)}$ or $[y,x]_{{\rm Lie}(G)}$ is a label of $G'$, then  ${\rm Lie}(G')$ is an ideal of ${\rm Lie}(G)$.
\end{enumerate}
\end{proposition}
\begin{proof}
\begin{enumerate}
\item Let $a,b\in V'$ be two connected vertices in $G'$. If $G'$ is a induced subgraph, then $[a,b]_{{\rm Lie}(G)}\neq 0$ and $[a,b]_{{\rm Lie}(G)}=\pm c(a,b)\in {\rm Lie}(G').$	
\item We need to prove that if $x$ is a vertex of $G'$ and $y$ is a vertex of $G$, then $[x,y]_{{\rm Lie}(G)}\in {\rm Lie}(G').$

\begin{itemize}
	\item If $x,y$ are not adjacent, then $[x,y]_{{\rm Lie}(G)}=0\in {\rm Lie}(G').$
	\item If $x,y$ are adjacent, then $[x,y]_{{\rm Lie}(G)}=\pm c(x,y)\in {\rm Lie}(G').$
\end{itemize}
Therefore, ${\rm Lie}(G')$ is an ideal of ${\rm Lie}(G).$\qedhere
\end{enumerate}		

\end{proof}

For the rest of the paper we will refer to the ideals in point $(2)$ above as graph-ideals.

\begin{cor}
Let $G=(V,E,c)$ be a labeled graph and $H$ a connected component of $G$. Then ${\rm Lie}(H)$ is a graph-ideal of ${\rm Lie}(G).$
\end{cor}

The following technical lemma gives a condition under which a change of the direction of a single directed edge of $G$ leads to isomorphic Lie algebras. Recall that the {\it neighborhood} $N(v)$ of a vertex $v$ of a directed graph $G$ is the set of all vertices of $G$ connected to or from $v$. 

\begin{lemma}\label{direction}
Let $G=(V,E,c)$ be a labeled directed graph and $G'$ the graph obtained changing the direction of an edge $\overrightarrow{ab}\in E$. Assume the edge $\overrightarrow{ab}$ is the only edge with label $c(a,b)\in{\mathcal C}$. Then ${\rm Lie}(G)$ is isomorphic to ${\rm Lie}(G').$
\end{lemma}

\begin{proof}
	
Define $C:V\times V\to \operatorname{span} \mathcal{C}$ by 
\[C(x,y)=\begin{cases} c(x,y) & \text{if}\quad \overrightarrow{xy}\in E\\
	-c(x,y) & \text{if}\quad \overrightarrow{yx}\in E\\
	0 & \text{if}\quad x\notin N(y)
\end{cases}\]
Consider the function $\varphi:{\rm Lie}(G)\to {\rm Lie}(G')$ by 
\[
\varphi(x)=\begin{cases}
	x & \mbox{if}\quad x\in V \quad \text{and}\quad x=a\\
	-x & \mbox{if}\quad x\in V \quad \text{and}\quad x\neq a\\
\end{cases}
\]
and 
\[\varphi(k)=\begin{cases}
	-k & \mbox{if}\quad x\in \mathcal{C} \quad \text{and}\quad k=\pm C(a,x), \quad x\neq b\quad \text{and}\quad x\in N(a)\\
	\,\,\,\, k & \text{otherwise}
\end{cases}\]
We need to prove that $\varphi[x,y]_{{\rm Lie}(G)}=[\varphi(x),\varphi(y)]_{{\rm Lie}(G')}$.

\begin{description}
\item[Case 1] $x=a$ and $y=b.$

\noindent Let $\alpha=c(a,b).$ Then $\varphi[a,b]_{{\rm Lie}(G)}=\varphi(\alpha)=\alpha$ and
\[\varphi[a,b]_{{\rm Lie}(G)}=[-a,b]_{{\rm Lie}(G')}=-[a,b]_{{\rm Lie}(G')}=[b,a]_{{\rm Lie}(G')}=\alpha.\]
\item[Case 2] $x=a$ and $y\in N(a)\setminus\{b\}$.

\noindent In this case
\[\varphi[a,y]_{{\rm Lie}(G)}=\varphi(C(a,y))=-C(a,y)\]
and \[[\varphi(a),\varphi(y)]_{{\rm Lie}(G')}=[-a,y]_{{\rm Lie}(G')}=-[a,y]_{{\rm Lie}(G')}=-C(a,y).\]
\item[Case 3] $x\notin N(a)$ and $x\notin N(a)$

In this case $[x,y]_{{\rm Lie}(G)}=0$ and $[\varphi(x),\varphi(y)]_{{\rm Lie}(G')}=0.$
\end{description}
We conclude that $\varphi:{\rm Lie}(G)\to {\rm Lie}(G')$ is an isomorphism.
\end{proof}

\begin{remark}
Notice that if there is another edge in $G$ with the same label as the edge $\overrightarrow{ab}\in E$, then the previous proof does not work. As can be easily seen from Example \ref{ex:1}, this is by no means a sufficient condition.
\end{remark}

\begin{cor}
Let $G=(V,E,c)$ be a labeled directed graph such that each $u\in \mathcal{C}$ appears at most once in each connected component of $G.$ Let $G'=(V,E',c)$ be the graph obtained inverting the direction of all the edges with the same label simultaneously. Then ${\rm Lie}(G)\cong {\rm Lie}(G').$\qedhere
\end{cor}
\begin{proof}
Define $\varphi$ as before in each connected component of $G.$
\end{proof}

The following is the main theorem of this paper. We use inductively the technical Lemma \ref{direction} to find a general criterion to perform direction changes and obtain isomorphic Lie algebras.

\begin{teo}\label{th:main}
Let $G=(V,E,c)$ be a labeled directed and connected graph such that $|\mathcal{C}|=|E|$, that is, each edge is uniquely labeled. Let $G'=(V,E',c)$ be the labeled directed graph obtained by changing the direction of any subset of edges in $G.$ Then ${\rm Lie}(G)\cong {\rm Lie}(G').$
\begin{proof}
Enumerate the edges that change direction by $\overrightarrow{a_1b_1},\overrightarrow{a_2b_2},\hdots, \overrightarrow{a_qb_q}.$ 

For $p=0,1,\dotsc, q$ define $G_p=(V,E_p,c)$ inductively by $G_0=G$ and $G_p$ as the graph obtained by changing the direction of $\overrightarrow{a_pb_p}$ in $G_{p-1}.$ Using the Proposition \ref{direction} we have ${\rm Lie}(G_p)\cong {\rm Lie}(G_{p-1})$ and by construction $G_q=G'.$ 

We conclude that  ${\rm Lie}(G)={\rm Lie}(G_0)\cong {\rm Lie}(G_1)\cong \cdots \cong {\rm Lie}(G_q)={\rm Lie}(G').$
\end{proof}
\end{teo}	

This theorem connects our construction to the one in \cite{mainkar1}. Since in that reference all labels are considered different, then the Lie algebras are defined regardless of the direction of the edges of the graph.

\begin{example}
The Lie algebras associated to the complete graphs $K_p$, with ${p\choose 2}$ different labels correspond to the free 2-step nilpotent Lie algebras generated by $p$ elements. See \cite{reu}.
\end{example}

\section{Graphs for Magnin's classification and their graph-ideals}\label{sec:magnin}

In \cite{magnin} it is possible to find a careful description of all real nilpotent Lie algebras of dimensions $\leq6$ and a classification of certain nilpotent algebras of dimension 7. The classification of real nilpotent algebras is still undergoing active research.

Applying directly the results from the previous section, we can present complete tables for the graphs of the 2-step nilpotent real Lie algebras shown in \cite{magnin} in dimension. For the sake of clarity, we divide this classification in dimensions 4 and 5, and dimension 6.

For brevity, we will not include as separate cases the inclusion of Abelian factors of ${\mathfrak g}_{-1}$, that is, those Abelian factors coming from disconnected vertices.

For the sake of notation, we denote by $\mathfrak h$ the 3-dimensional Heisenberg Lie algebra and by ${\mathfrak g}_1$ the 1-dimensional Abelian Lie algebra. The Lie algebras $\mathfrak{g}_{5,1}$ and $\mathfrak{g}_{5,2}$ in Table \ref{dim45} appear named as such in \cite{magnin}. 

\begin{table}[h]
\caption{Graphs, non-trivial subalgebras and non-trivial graph-ideals in dimensions 4 and 5}\label{dim45}
\begin{tabular}{|c|c|c|c|c|}
\hline
Dimension& Lie algebra & Graph & Subalgebras& Graph-ideals \\ \hline
4& $\mathfrak{h}\times \mathfrak{g}_1$ & 	\begin{minipage}{3cm}\includegraphics[scale=0.4]{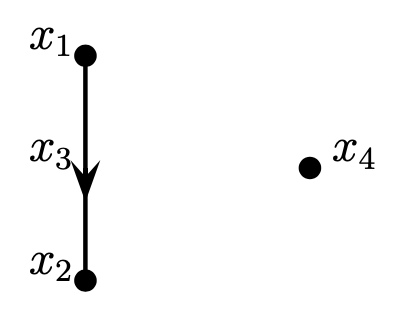}\end{minipage} & $\langle x_1,x_2,x_3\rangle$ & $\langle x_1,x_2,x_3\rangle$ \\ \hline
\multirow{11}{*}{5} & $\mathfrak{g}_{5,1}$ & \begin{minipage}{3cm}\vspace{0.1cm}\includegraphics[scale=0.4]{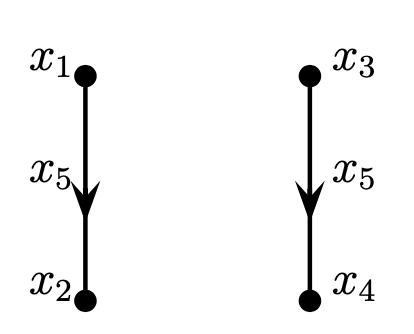}\vspace{0.1cm}\end{minipage} & \begin{minipage}{2cm}$\langle x_1,x_2,x_5\rangle$\\$\langle x_3,x_4,x_5\rangle$\end{minipage} &
\begin{minipage}{2cm}$\langle x_1,x_2,x_5\rangle$\\$\langle x_3,x_4,x_5\rangle$\end{minipage}\\ \cline{2-5}
&$\mathfrak{g}_{5,2}$  & \begin{minipage}{3cm}\vspace{0.1cm}\includegraphics[scale=0.4]{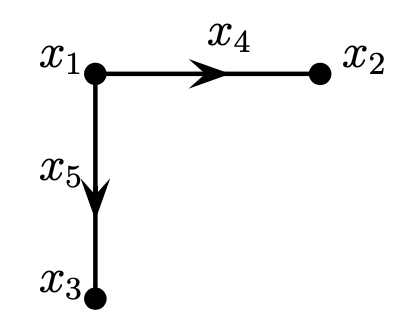}\end{minipage} & 
\begin{minipage}{2cm}$\langle x_1,x_3,x_5\rangle$\\$\langle x_1,x_2,x_4\rangle$\end{minipage}& --- \\\cline{2-5}
 & $\mathfrak{h}\times \mathfrak{g}_1^2$  & \begin{minipage}{3cm}\vspace{0.1cm}\includegraphics[scale=0.4]{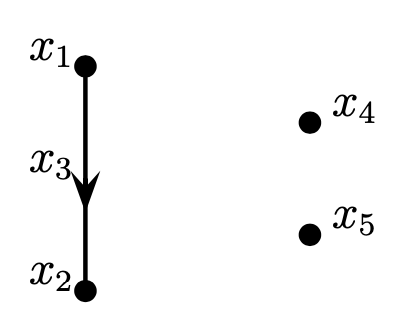}\vspace{0.1cm}\end{minipage} & $\langle x_1,x_2,x_3\rangle$ & $\langle x_1,x_2,x_3\rangle$ \\ \hline
\end{tabular}
\end{table}

The Lie algebras $\mathfrak{g}_{6,1}$, $\mathfrak{g}_{6,2}$ and $\mathfrak{g}_{6,3}$ in Table \ref{dim6} only appear listed in \cite{magnin}, and we named them as such to make both tables coherent. The Lie algebra $\mathfrak{g}_{6,3}$ does not exactly fit our theory, since it corresponds to a family of Lie algebras depending on a parameter $\gamma\neq0$ and not a square. As such, we do not mention either subalgebras nor graph-ideals, but we add it to the table for completeness.

\begin{table}[h]
\caption{Graphs, non-trivial subalgebras and non-trivial graph-ideals in dimension 6}\label{dim6}
\begin{tabular}{|c|c|c|c|}
\hline
Lie algebra & Graph & Subalgebras & Graph-ideals \\ \hline
$\mathfrak{g}_{6,1}$ & \begin{minipage}{3cm}\includegraphics[scale=0.4]{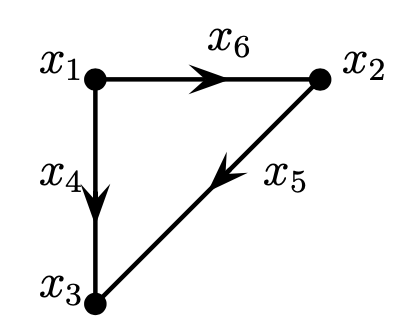}\end{minipage} &\begin{minipage}{2cm}$\langle x_1,x_3,x_4\rangle$\\$\langle x_1,x_2,x_6\rangle$\\$\langle x_2,x_3,x_5\rangle$\end{minipage} & ---
 \\ \hline
$\mathfrak{g}_{6,2}$ & \begin{minipage}{3cm}\vspace{0.1cm}\includegraphics[scale=0.4]{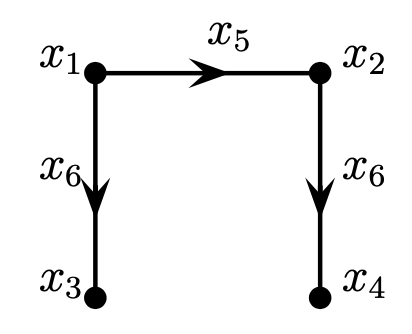}\vspace{0.1cm}\end{minipage} & \begin{minipage}{3.2cm}\begin{center}$\langle x_1,x_2,x_6\rangle$\\$\langle x_1,x_2,x_5\rangle$\\$\langle x_2,x_4,x_6\rangle$\\$\langle x_1,x_2,x_3,x_5,x_6\rangle$\\$\langle x_1,x_2,x_4,x_5,x_6\rangle$\end{center}\end{minipage} &
\begin{minipage}{3.2cm}$\langle x_1,x_2,x_3,x_5,x_6\rangle$\\$\langle x_1,x_2,x_4,x_5,x_6\rangle$\end{minipage}\\ \hline 
\begin{minipage}{2cm}\begin{center}$\mathfrak{g}_{6,3}$\\{\footnotesize$\gamma\neq0,\alpha^2$}\end{center}\end{minipage} & \begin{minipage}{3.6cm}\vspace{0.1cm}\includegraphics[scale=0.35]{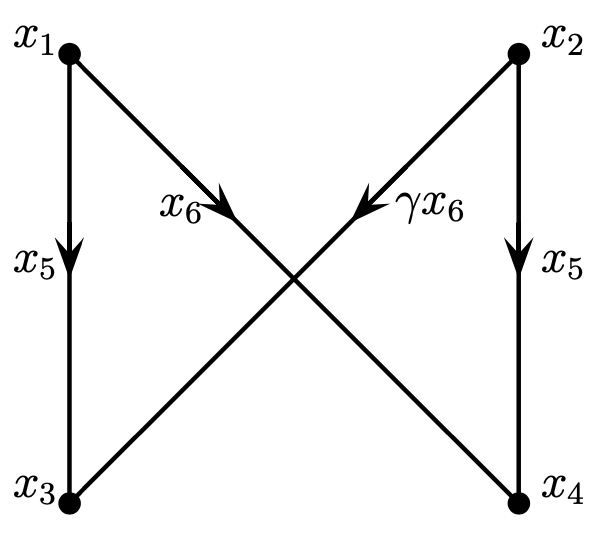}\vspace{0.1cm}\end{minipage} & --- & --- \\\hline
$\mathfrak{h}\times\mathfrak{h}$  &  	\begin{minipage}{3cm}\vspace{0.1cm}\includegraphics[scale=0.4]{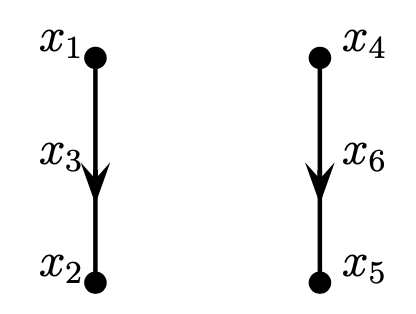}\end{minipage} &
\begin{minipage}{2cm}$\langle x_1,x_2,x_3\rangle$\\$\langle x_4,x_5,x_6\rangle$\end{minipage}&
\begin{minipage}{2cm}$\langle x_1,x_2,x_3\rangle$\\$\langle x_4,x_5,x_6\rangle$\end{minipage}\\ \hline
$\mathfrak{h}\times\mathfrak{g}_1^3$  &  	\begin{minipage}{3cm}\vspace{0.1cm}\includegraphics[scale=0.4]{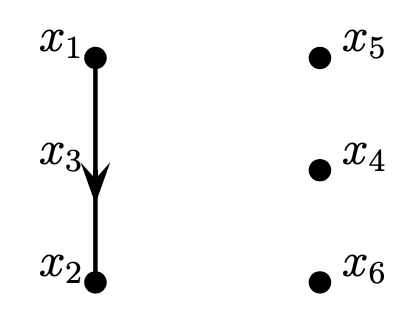}\end{minipage}& 
$\langle x_1,x_2,x_3\rangle$ & $\langle x_1,x_2,x_3\rangle$ \\ \hline
$\mathfrak{g}_{5,1}\times\mathfrak{g}_1$  &  	\begin{minipage}{4.5cm}\vspace{0.1cm}\includegraphics[scale=0.4]{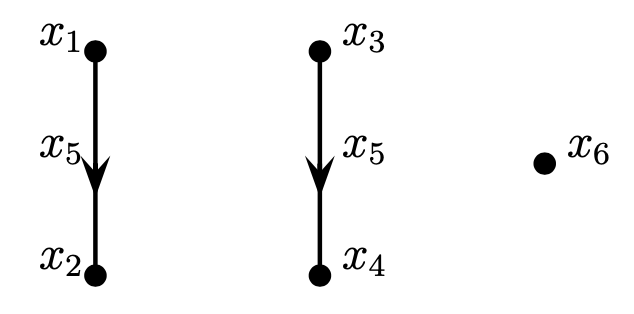}\end{minipage}& 
\begin{minipage}{3.2cm}\begin{center}$\langle x_1,x_2,x_5\rangle$\\$\langle x_3,x_4,x_5\rangle$\\$\langle x_1,x_2,x_3,x_4,x_5\rangle$\end{center}\end{minipage}& 
\begin{minipage}{3.2cm}\begin{center}$\langle x_1,x_2,x_5\rangle$\\$\langle x_3,x_4,x_5\rangle$\\$\langle x_1,x_2,x_3,x_4,x_5\rangle$\end{center}\end{minipage}\\ \hline
$\mathfrak{g}_{5,2}\times\mathfrak{g}_1$  &  	\begin{minipage}{3cm}\vspace{0.1cm}\includegraphics[scale=0.4]{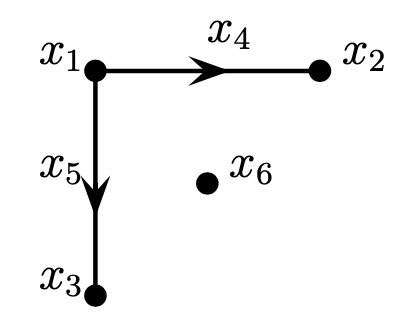}\vspace{0.2cm}\end{minipage} & \begin{minipage}{3.2cm}\begin{center}$\langle x_1,x_3,x_5\rangle$\\$\langle x_1,x_2,x_4\rangle$\\$\langle x_1,x_2,x_3,x_4,x_5\rangle$\end{center}\end{minipage}&
$\langle x_1,x_2,x_3,x_4,x_5\rangle$\\ \hline
\end{tabular}
\end{table}

\section{The algebra of graded derivations for $K_{m,n}$}\label{sec:Kmn}

When studying Lie algebras in different contexts, it is common to ask oneself firstly about isomorphisms, and immediately after about derivations. If the Lie algebra ${\mathfrak g}$ being studied is graded, then the natural question is to try to compute the space of graded derivations ${\rm Der}_0({\mathfrak g})$. In this section, we present some computations associated to the complete bipartite graph $K_{m,n}$. As it will be mentioned a the end of this section, this space is related to infinitesimal symmetries in the case of real and complex Lie algebras, according to the theory of Tanaka prolongation, see \cite{tanaka}.

Recall that the complete bipartite graph $K_{m,n}=(V,E)$ is the undirected graph in which all vertices from the set $X=\{x_{1},\hdots,x_m\}$ are connected to all vertices from the set $Y=\{y_{1},\hdots,y_{n}\}$. As such, the set $V$ is the disjoint union of $X$ and $Y$, and $E$ contains $mn$ edges. 

In what follows, we will direct all edges going from $X$ to $Y$. In subsection \ref{ssec:kmn}, this orientation is irrelevant, according to Theorem \ref{th:main}.

\subsection{Labeling $K_{m,n}$ with a single label}

\noindent Let $G_{1}$ be the complete bipartite graph with edges 
\[
E=\{\overrightarrow{xy}\mid x\in X,\quad y\in Y\}
\]
and labeled by a unique label $u$. As before, recall that ${\rm Lie}(G_1)$ is naturally stratified as ${\mathfrak g}_{-2}\oplus{\mathfrak g}_{-1}$, where ${\mathfrak g}_{-1}={\rm span}_F(X\cup Y)$ and ${\mathfrak g}_{-2}={\rm span}_F\{u\}$.

In order to study the space of graded derivations of the Lie algebra associated to $G_1$, it is convenient to think of an element $\varphi\in{\rm Der}_0({\rm Lie}(G_1))$ as a matrix in the obvious basis, namely
\[
\varphi=\left(\begin{array}{c|c}\begin{array}{c|c}\alpha&\gamma\\\hline\beta&\delta\end{array}&0\\\hline0&\lambda\end{array}\right)
\]

where we have the linear maps
\begin{eqnarray*}
&\alpha\colon{\rm span}_F(X)\to{\rm span}_F(X),\quad\beta\colon{\rm span}_F(X)\to{\rm span}_F(Y),&\\
&\gamma\colon{\rm span}_F(Y)\to{\rm span}_F(X),\quad\delta\colon{\rm span}_F(Y)\to{\rm span}_F(Y),&
\end{eqnarray*}
and $\lambda\in F$ is a scalar.

\begin{proposition}
The dimension of ${\rm Der}_0({\rm Lie}(G_1))$ is 
\[
\dfrac{(m+n)(m+n+1)}{2}+1.
\]
\end{proposition}

\begin{proof}
Let $\varphi\in {\rm Der}_0({\rm Lie}(G_1))$. Then there are constants $\lambda,\alpha_{ri},\beta_{si},\gamma_{rj},\delta_{sj}\in F$ (where $r,i=1,\dotsc,m$ and $s,j=1,\dotsc,n$), such that
\[
\varphi(x_i)=\sum_{r=1}^m \alpha_{ri}x_r+\sum_{s=1}^n \beta_{si}y_{s},\quad \varphi(y_j)=\sum_{r=1}^m \gamma_{rj}x_r+\sum_{s=1}^n \delta_{sj}y_{s},\quad\varphi(u)=\lambda u.
\]
Thus, for $i,i'=1,\dotsc,m$ and $j,j'=1,\dotsc,n$, we have
\begin{align*}
\varphi[x_i,x_{i'}]_{{\rm Lie}(G_1)}&=[\varphi(x_i),x_{i'}]_{{\rm Lie}(G_1)}+[x_i,\varphi(x_{i'})]_{{\rm Lie}(G_1)}=\left(-\sum_{s=1}^n \beta_{si}+\sum_{s=1}^n \beta_{si'}\right)u=0,\\
\varphi[y_j,y_{j'}]_{{\rm Lie}(G_1)}&=[\varphi(y_j),y_{j'}]_{{\rm Lie}(G_1)}+[y_j,\varphi(y_{j'})]_{{\rm Lie}(G_1)}=\left(\sum_{r=1}^m \gamma_{rj}-\sum_{r=1}^m \gamma_{rj'}\right)u=0,\\
\varphi[x_i,y_j]_{{\rm Lie}(G_1)}&=[\varphi(x_i),y_j]_{{\rm Lie}(G_1)}+[x_i,\varphi(y_j)]_{{\rm Lie}(G_1)}=\left(\sum_{r=1}^m\alpha_{ri}+\sum_{s=1}^n \delta_{sj}\right)u=\lambda u.
\end{align*}

As a consequence, we have the following homogeneous system 
\[
\displaystyle{-\sum_{s=1}^n \beta_{si}+\sum_{s=1}^n \beta_{si'}=0},\quad 
\displaystyle{\sum_{r=1}^m \gamma_{rj}-\sum_{r=1}^m \gamma_{rj'}=0},\quad
\displaystyle{\sum_{r=1}^m \alpha_{ri}+\sum_{s=1}^n \delta_{sj}-\lambda=0}
\]
of $\dfrac{n(n-1)}{2}+\dfrac{m(m-1)}{2}+mn$ independent linear equations. Considering the following simple computation
\[
(m+n)^2+1-\dfrac{m(m-1)}{2}-\dfrac{n(n-1)}{2}-mn=\dfrac{(m+n)(m+n+1)}{2}+1.
\]
we conclude with the expected result.
\end{proof}

\subsection{Labeling $K_{m,n}$ with $mn$ different labels}\label{ssec:kmn}

Let $G_2$ be the complete bipartite graph with edges 
\[
E=\{\overrightarrow{xy}\mid x\in X,\quad y\in Y\}
\]
and labeled by the set ${\mathcal C}=\{c_{ij}\mid 1\leq i \leq m, 1\leq j\leq n\}$. We will consider that all labels are different, that is, we assume ${\mathcal C}$ has $mn$ elements. Once again, we stratify ${\rm Lie}(G_2)$ as ${\mathfrak g}_{-2}\oplus{\mathfrak g}_{-1}$, where ${\mathfrak g}_{-1}={\rm span}_F(X\cup Y)$ and ${\mathfrak g}_{-2}={\rm span}_F({\mathcal C})$.

As in the previous subsection, we write an element $\varphi\in{\rm Der}_0({\rm Lie}(G_2))$ as a matrix in the obvious basis, namely
\[
\varphi=\left(\begin{array}{c|c}\begin{array}{c|c}\alpha&\gamma\\\hline\beta&\delta\end{array}&0\\\hline0&\epsilon\end{array}\right)
\]
where we have the linear maps
\begin{eqnarray*}
&\alpha\colon{\rm span}_F(X)\to{\rm span}_F(X),\quad\beta\colon{\rm span}_F(X)\to{\rm span}_F(Y),&\\
&\gamma\colon{\rm span}_F(Y)\to{\rm span}_F(X),\quad\delta\colon{\rm span}_F(Y)\to{\rm span}_F(Y),&\\
&\epsilon\colon{\rm span}_F({\mathcal C})\to{\rm span}_F({\mathcal C}).&
\end{eqnarray*}

\begin{proposition}
The dimension of ${\rm Der}_0({\rm Lie}(G_2))$ is $m^2+n^2+m^2n^2-mn$.
\end{proposition}

\begin{proof}
Let $\varphi\in{\rm Der}_0({\rm Lie}(G_2))$. Then there are constants $\alpha_{ri},\beta_{si},\gamma_{rj},\delta_{sj},\epsilon_{ij}^{rs}\in F$ (where $r,i=1,\dotsc,m$ and $s,j=1,\dotsc,n$), such that 

\begin{align*}
\varphi(x_i)=\sum_{r=1}^m \alpha_{ri}x_r+\sum_{s=1}^n \beta_{si}y_s,\quad 
\varphi(y_j)=\sum_{r=1}^m \gamma_{rj}x_r+\sum_{s=1}^n \delta_{sj}y_s,\quad 
\varphi(c_{ij})=\sum_{r=1}^m \sum_{s=1}^n \epsilon_{ij}^{rs} c_{rs}.
\end{align*}
Thus, for $i,i'=1,\dotsc,m$ and $j,j'=1,\dotsc,n$, we have
\begin{align*}
0=\varphi[x_i,x_{i'}]_{{\rm Lie}(G_2)}&=[\varphi(x_i),x_{i'}]_{{\rm Lie}(G_2)}+[x_i,\varphi(x_{i'})]_{{\rm Lie}(G_2)}=-\sum_{s=1}^n \beta_{si}c_{i's}+\sum_{s=1}^n \beta_{si'}c_{is},\\
0=\varphi[y_j,y_{j'}]_{{\rm Lie}(G_2)}&=[\varphi(y_j),y_{j'}]_{{\rm Lie}(G_2)}+[y_j,\varphi(y_{j'})]_{{\rm Lie}(G_2)}=\sum_{r=1}^m\gamma_{rj}c_{rj'}-\sum_{r=1}^m\gamma_{rj'}c_{rj}.
\end{align*}
If $i=i'$, then the first equality reduces to
\[
-\sum_{s=1}^n \beta_{si}c_{i's}+\sum_{s=1}^n \beta_{si'}c_{is}=-\sum_{s=1}^n \beta_{si}c_{is}+\sum_{s=1}^n \beta_{si}c_{is}=0,
\]
that is, it is satisfied trivially for any $i=1,\hdots,m$. 

On the other hand, if $i\neq i'$ then $c_{i's}\neq c_{is}$ and as ${\mathcal C}=\{c_{ij}\}$ is a linearly independent set, then $\beta_{si}=\beta_{si'}=0$ for all $i,i'=1,\hdots, m$ and $s=1,\hdots, n.$ Similarly, we obtain $\gamma_{rj}=\gamma_{rj'}=0$ for all $j,j'=1,\hdots, n$ and $r=1,\hdots, m.$

Note that 
\[\varphi(c_{ij})=\sum_{r=1}^m \sum_{s=1}^n \epsilon_{ij}^{rs}c_{rs}=[\varphi(x_i),y_j]_{{\rm Lie}(G_2)}+[x_i,\varphi(y_j)]_{{\rm Lie}(G_2)}=\sum_{r=1}^m\alpha_{ri}c_{rj}+\sum_{s=1}^n \delta_{sj}c_{is}\]
and we have $mn$ linearly independent equations. The result follows.
\end{proof}

\begin{remark}
For a real or complex stratified Lie algebra $\mathfrak{n}=\mathfrak{n}_{-s}\oplus \cdots \oplus \mathfrak{n}_{-1}$, there is an important object called its \emph{Tanaka prolongation}
\[
{\rm Prol}(\mathfrak{n})=\mathfrak{n}_{-s}\oplus \cdots \oplus \mathfrak{n}_{-1}\oplus\mathfrak{n}_{0}\oplus\mathfrak{n}_{1}\oplus\cdots
\]
which helps computing infinitesimal symmetries of certain differential systems, see \cite{tanaka}. Each of the subspaces $\mathfrak{n}_r$ for $r\geq0$ can be computed explicitly using a simple induction procedure. In particular, it follows directly from the construction that
\[
\mathfrak{n}_0={\rm Der}_0(\mathfrak{n}).
\]
This equality has been successfully in \cite{ben} to completely characterize the Tanaka prolongations of free Lie algebras.
\end{remark}

\section{Acknowledgments} 
The second author would like to thank professor Mar\'ia Alejandra \'Alvarez from Universidad de Antofagasta, Chile, for her hospitality and her comments regarding earlier versions of this paper. Both authors would like to thank professor Andrew Clarke from Universidade Federal do Rio de Janeiro, Brazil, for his hospitality.

\end{document}